\documentclass[a4paper]{amsart} 

\usepackage[all]{xy}
\usepackage{amssymb}
\usepackage{eucal}
\usepackage{hyperref}
\usepackage{tikz}
\usetikzlibrary{arrows}

\calclayout
\SelectTips{eu}{}

\theoremstyle{definition}
\newtheorem{theorem}{Theorem}[section]
\newtheorem{lemma}[theorem]{Lemma}
\newtheorem{prop}[theorem]{Proposition}
\newtheorem{defn}[theorem]{Definition}
\newtheorem{cor}[theorem]{Corollary}
\newtheorem{conj}[theorem]{Conjecture}
\newtheorem{rem}[theorem]{Remark}
\newtheorem{ex}[theorem]{Example}

\title{Acyclic cluster algebras from a ring theoretic point of view}
\author{Philipp Lampe}
\address{Philipp Lampe\\ Fakult\"at f\"ur Mathematik\\
Universit\"at Bielefeld\\ D-33501 Bielefeld\\ Germany.}
\email{lampe@math.uni-bielefeld.de}

\thanks{Version from \today.}

\begin{document}
\maketitle
\begin{abstract}The article gives a ring theoretic perspective on cluster algebras.  Gei\ss-Leclerc-Schr\"oer prove that all cluster variables in a cluster algebra are irreducible elements. Furthermore, they provide two necessary conditions for a cluster algebra to be a unique factorization domain, namely the irreducibility and the coprimality of the initial exchange polynomials. 

We present a sufficient condition for a cluster algebra to be a unique factorization domain in terms of primary decompositions of certain ideals generated by initial cluster variables and initial exchange polynomials. As an application, the criterion enables us to decide which coefficient-free cluster algebras of Dynkin type $A, D$ or $E$ are unique factorization domains. Moreover, it yields a normal form for irreducible elements in cluster algebras that satisfy the condition. Proof techniques include methods from commutative algebra. 

In addition, we state a conjecture about the range of application of the criterion.
\end{abstract}

\maketitle
\setcounter{tocdepth}{1}
\tableofcontents

\section{Introduction}
The article studies Fomin-Zelevinsky's \textit{cluster algebras} from a ring-theoretic point of view. We focus on two questions: When is a given cluster algebra a unique factorization domain? What are irreducible elements?

In a seminal paper Fomin-Zelevinsky \cite{FZ} have introduced cluster algebras in order to study Lusztig's \textit{canonical bases} \cite{L1} arising in Lie theory and \textit{totally positive matrices} \cite{L2}. The theory of cluster algebras, which was further developed by Fomin-Zelevinsky \cite{FZ2,FZ4} and Berenstein-Fomin-Zelevinsky \cite{BFZ3}, has amplified further branches of mathematics. Let us mention two directions: Via the \textit{Caldero-Chapoton map} \cite{CC} cluster algebras have become a tool to study \textit{representations of quivers}. The theory culminated in the categorification of acyclic cluster algebras by Buan-Marsh-Reineke-Reiten-Todorov's \textit{cluster categories} \cite{BMRRT} by work of the same authors \cite{BMRRT} and Caldero-Keller \cite{CK,CK2}. As a second example, Fomin-Shapiro-Thurston \cite{FST} (generalizing work of Fock-Goncharov \cite{FG1,FG2} and Gekhtman-Shapiro-Vainshtein \cite{GSV}) have constructed cluster algebras from marked surfaces using Poisson geometry.

The definition of a cluster algebra is elementary. A cluster algebra of rank $n$ contains a distinguished set of generators called the set of \textit{cluster variables} and the cluster variables clump into several sets of cardinality $n$ which may overlap and which are called \textit{clusters}. Any cluster can be obtained from a reference cluster, which is called the \textit{initial cluster}, by a sequence of \textit{mutations}. Every mutation exchanges exactly one cluster variable of a cluster by another. There are explicit combinatorial formulae for the exchange relations, so that the initial cluster together with the initial exchange data, which is usually stored in a matrix denoted by the letter $B$, determines the whole cluster algebra.  

Many authors have studied the structure of a cluster algebra as a \textit{vector space}. Examples for vector space bases include the set of \textit{cluster monomials} for coefficient-free cluster algebras of Dynkin type due to Caldero-Keller \cite{CK}, Lee-Li-Zelevinsky's \textit{greedy basis} \cite{LLZ} for cluster algebras of rank 2, the \textit{dual semicanonical basis} for cluster algebras attached to unipotent cells due to Gei\ss-Leclerc-Schr\"oer \cite{GLS3,GLS2}, Dupont-Thomas's \textit{atomic basis} \cite{DT}, and Musiker-Schiffler-Williams's \textit{bangle} and \textit{bracelet basis} \cite{MSW} for cluster algebras from surfaces. Moreover, Cerulli-Keller-Labardini-Plamondon \cite{CKLP} have shown that the set of cluster monomials in any skew-symmetric cluster algebra is linearly independent. 

The study of bases and linear independence is a natural question to ask since canonical bases were a motivation for introducing cluster algebras. On the other hand, in most examples the algebraic structure is essential: algebraic exchange relations in the cluster algebra correspond to \textit{exchange triangles} in the context of cluster categories and to \textit{Ptolemy relations} in the context of marked surfaces. Therefore, it also natural to study the structure of a given cluster algebra as a \textit{ring}. 

By construction every cluster algebra is an integral domain as it is a subring of its ambient field. It is seldom a principal ideal domain. In between there is the class of unique factorization domains: every principal ideal domain is a unique factorization domain and every unique factorization domain is an integral domain. Therefore, the question which cluster algebras are unique factorization domains suggests itself. Gei\ss-Leclerc-Schr\"oer \cite[Section 6.1]{GLS} give two necessary conditions for a cluster algebra (of rank $n$) to be a unique factorization domain. The key players in both criteria are certain polynomials in the initial cluster variables which describe the initial exchange relations and which we shall refer to as the \textit{initial exchange polynomials}. The criteria say that if a cluster algebra is a unique factorization domain, then its initial exchange polynomials are irreducible and pairwise different. In this article, we state a sufficient condition for an \textit{acyclic} cluster algebra to be a unique factorization domain, see Theorem \ref{Criterion}. Moreover, we conjecture that for acyclic cluster algebras the existence of a seed, with respect to which some exchange polynomial is reducible or two exchange polynomials have a common factor, is the only obstruction to unique factorization. The conjecture would follow from another more algebraic conjecture, see Conjecture \ref{IdealProperties}, which we verify in several cases.       

Let us say a few words about the acyclicity assumption. We shall give a precise definition in Section \ref{Recap}, a way to think about the assumption is as follows: An obvious generating set of a cluster algebra is by definition the set of cluster variables. The set of cluster variables can be finite or infinite, and Fomin-Zelevinsky \cite[Theorem 1.4]{FZ2} classified cluster algebra of finite type. As a generating set, the set of cluster variables often is redundant, and even cluster algebras of infinite type are often generated by finitely many elements. The acyclicity assumption ensures that this is true. More precisely, an acyclic coefficient-free cluster algebra of rank $n$ is generated by $2n$ elements, namely the cluster variables of an acyclic seed together with the ones obtained from it by one mutation. As finitely generated algebras, acyclic cluster algebras are Noetherian. Non-acyclic cluster algebras can however be non-Noetherian and not finitely generated, see Muller \cite[Proposition 11.3]{M}.

The article is organized as follows. In the next section, we introduce the basic notions in the context of cluster algebras. More precisely, Section \ref{BN1} focuses on the definitions, Section \ref{BN2} recalls Berenstein-Fomin-Zelevinsky's main results \cite{BFZ3} on lower and upper bounds, Section \ref{BN3} presents Gei\ss -Leclerc-Schr\"oer's results \cite{GLS} on irreducible elements and unique factorization, and Section~\ref{BN4} illustrates the results for skew-symmetric cluster algebras of finite type.  

We devote Section \ref{Description} to the aforementioned sufficient criterion for unique factorization. We formulate precise assumptions under which we expect unique factorization -- informally already stated above -- in Subsection \ref{Assumptions}. Subsection \ref{Ideals} introduces certain prime ideals in the polynomial ring generated by the initial cluster variables and studies their algebraic properties. We explain our motivation for the definition of the ideals in Subsection \ref{des}: using Fomin-Zelevinsky's Laurent phenomenon \cite[Theorem 3.1]{FZ} and Berenstein-Fomin-Zelevinsky's equality of an acyclic cluster algebra and its lower bound \cite[Theorem 1.20]{BFZ3} we see that polynomials appearing in numerators of cluster algebra elements lie in products of the just introduced ideals. In Subsection \ref{conjecture} we formulate a conjecture for the primary decomposition of such product ideals. In Subsection \ref{Crit} we state the theorem relating the criterion with unique factorization.       

Having proved this theorem, we are ready to prove unique factorization in different cases in Section \ref{Applications}. Besides cluster algebras of rank 2, to be handled in Subsection \ref{RankTwo}, the section concerns mainly skew-symmetric cluster algebras of finite type. Subsection \ref{Lemmata} states two auxiliary lemmas. Subsection \ref{DynkinUFD} verifies Conjecture \ref{IdealProperties} for cluster algebras of simply-laced Dynkin type. We summarize the surprising results in the following table. 

\begin{center}
\begin{tabular}{|c||c|c|}\hline
Type&UFD&not UFD\\\hline\hline
$A_n$&$n\neq 3$&$n=3$\\\hline
$D_n$&-&$n\geq 4$\\\hline
$E_n$&$n=6,7,8$&-\\\hline
\end{tabular}
\end{center}

\section{Reminder: the basic notions of cluster algebras}
\label{Recap}

\subsection{The definition of a cluster algebra and basic properties}
\label{BN1}
In this section we briefly outline the definition of cluster algebras and present their main features. Cluster algebras are commutative algebras; in their seminal paper \cite{FZ} Fomin-Zelevinsky use the ring of integers as a ground ring. More generally, it is also possible to define cluster algebras with an arbitrary field as ground ring as long as the characteristic is zero, see Gei\ss -Leclerc-Schr\"oer \cite{GLS} for example. In this presentation we stick to the more general setup.

Let $K$ be a field of characteristic $\textrm{char}(K)=0$ or let $K=\mathbb{Z}$. Furthermore, let $m$ and $n$ be positive integers with $m\geq n\geq 1$. We fix algebraically independent variables $u_1,\ldots,u_m$ and call the field $\mathcal{F}=K(u_1,\ldots,u_m)$ of rational functions in these variables the \textit{ambient field}. By the term \textit{algebra} we mean a unital algebra with a multiplicative identity element.

An integer $n\times n$ matrix $B$ is called \textit{skew-symmetrizable} if there is a diagonal matrix $D=\textrm{diag}(d_1,d_2,\ldots,d_n)$ with positive integer diagonal entries such that the matrix $DB$ is skew-symmetric, i.e., $d_ib_{ij}=-d_jb_{ji}$ for all $1\leq i,j\leq n$. Let us mention some properties: If $B$ is a skew-symmetrizable integer $n\times n$ matrix, then for all $1\leq i\leq n$ the diagonal entry $b_{ii}$ is zero and for all $1\leq i,j\leq n$ the entry $b_{ij}$ is zero if and only if the entry $b_{ji}$ is zero. Moreover, $B$ is \textit{sign skew-symmetric}, i.e., if $b_{ij}>0$ for some indices $1\leq i,j\leq n$, then the we have $b_{ji}<0$.

Furthermore, if $\tilde{B}$ is an integer $m\times n$ matrix, then we call the $n\times n$ submatrix $B$ formed by the first $n$ rows the \textit{principal part} of $\tilde{B}$. 

An \textit{exchange matrix} $\tilde{B}$ is an integer $m\times n$ matrix $\tilde{B}$ with a skew-symmetrizable principal part $B$. A \textit{cluster} is  a sequence $\mathbf{x}=(x_1,x_2,\ldots x_m)\in \mathcal{F}^m$ of $m$ elements in $\mathcal{F}$ which are algebraically independent over $K$. A {\it seed} is a pair $(\mathbf{x},\tilde{B})$ consisting of a cluster $\mathbf{x}$ and an exchange matrix $\tilde{B}$. The first $n$ elements $x_1,\ldots,x_n$ of $\mathbf{x}$ are called \textit{cluster variables}, the last $m-n$ elements $x_{n+1},\ldots,x_m$ of $\mathbf{x}$ of $\mathbf{x}$ are called \textit{frozen variables}. 

Let $(\mathbf{x},\tilde{B})$ be a seed and let $k$ be an integer such that $1 \leq k \leq n$. The \textit{mutation} of $(\mathbf{x},\tilde{B})$ at $k$ is a new seed $\mu_k(\mathbf{x},\tilde{B})=(\mathbf{x}',\tilde{B}')$ with exchange matrix $\tilde{B}' = (b_{ij}')$ defined by
\begin{align*}
b_{ij}'=
\begin{cases}
- b_{ij} & \text{if $i=k$ or $j=k$},\\
b_{ij} + \dfrac{|b_{ik}|b_{kj} + b_{ik}|b_{kj}|}{2} &\text{otherwise},
\end{cases}
\end{align*}
and a new cluster defined $\mathbf{x}' = (x_1',\ldots,x_m')$ defined by
\begin{align*}
x_j' =\begin{cases}
x_k^{-1}\left(\prod_{b_{ik} > 0} x_i^{b_{ik}} +\prod_{b_{ik} < 0} x_i^{-b_{ik}}\right) &\text{if $j=k$},\\
x_j & \text{otherwise}.
\end{cases}
\end{align*}
The last relation is called \textit{exchange relation}. It is easy to see that $\mu_k(\mathbf{x},\tilde{B})$ is again a seed (with the very same diagonal matrix $D$ proving the skew-symmetrizability of the principle part). We also denote $\tilde{B'}$ by $\mu_k(\tilde{B})$ and $\mathbf{x}'$ by $\mu_k(\mathbf{x})$. The following proposition establishes  a main property of mutations; it is due to Fomin-Zelevinsky \cite[Section 4]{FZ}.

\begin{prop}
\label{involution}
The map $\mu_k$ is an involution, i.e., $(\mu_k\circ\mu_k)(\mathbf{x},\tilde{B}) = (\mathbf{x},\tilde{B})$ for all $1\leq k\leq n$ and all seeds $(\mathbf{x},\tilde{B})$.
\end{prop}

Two seeds $(\mathbf{x},\tilde{B})$ and $(\mathbf{y},\tilde{C})$ are called \textit{mutation equivalent} if there is a sequence $(k_1,\ldots,k_r)$ of indices with $1\leq k_j\leq n$ for all $j$ such that $(\mu_{k_1}\circ\mu_{k_2}\circ\ldots\circ\mu_{k_r})(\mathbf{x},\tilde{B})=(\mathbf{y},\tilde{C})$. If $(\mathbf{x},\tilde{B})$ and $(\mathbf{y},\tilde{C})$ are two mutation equivalent seeds, then we also write $(\mathbf{x},\tilde{B})\sim(\mathbf{y},\tilde{C})$. By Proposition \ref{involution} this defines an equivalence relation on the set of all seeds.    

Let us fix a particular seed $(\mathbf{x},\tilde{B})$ which we refer to as \textit{initial seed}. Likewise, we call $\mathbf{x}$ the \textit{initial cluster}, the elements $x_1,\ldots,x_n$ of $\mathbf{x}$ \textit{initial cluster variables}, and $\tilde{B}$ the \textit{initial exchange matrix}. Put 
\begin{align*}
\mathcal{X}{(\mathbf{x},\tilde{B})}= \bigcup_{(\mathbf{y},\tilde{C}) \sim (\mathbf{x},\tilde{B})} \{y_1,\ldots,y_n\}
\end{align*} 
where the union is taken over all seeds $(\mathbf{y},\tilde{C})$ mutation equivalent to $(\mathbf{x},\tilde{B})$. The \textit{cluster algebra} $\mathcal{A}(\mathbf{x},\tilde{B})$ (of geometric type) associated with the seed $(\mathbf{x},\tilde{B})$ is defined as the $K$-subalgebra of $\mathcal{F}$ generated by 
\begin{align}
\label{GeneratingSet}
\mathcal{X}{(\mathbf{x},\tilde{B})}\cup\{x_{n+1}^{\pm 1},\ldots,x_{m}^{\pm 1}\}.
\end{align}
Put $L=K[x_{n+1}^{\pm 1},\ldots,x_m^{\pm 1}]$. By definition, $\mathcal{A}(\mathbf{x},\tilde{B})$ is naturally an $L$-algebra. As an $L$-algebra, it is generated by $\mathcal{X}{(\mathbf{x},\tilde{B})}$. 

We call the elements in $\mathcal{X}(\mathbf{x},\tilde{B})$ the \textit{cluster variables} of $\mathcal{A}(\mathbf{x},\tilde{B})$ and $x_{n+1},\ldots,x_m$ the \textit{frozen variables} of  $\mathcal{A}(\mathbf{x},\tilde{B})$. Furthermore, if $(\mathbf{y},\tilde{C}) \sim (\mathbf{x},\tilde{B})$, then we call $(\mathbf{y},\tilde{C})$ a \textit{seed} of $\mathcal{A}(\mathbf{x},\tilde{B})$, $\mathbf{y}$ a \textit{cluster} of $\mathcal{A}(\mathbf{x},\tilde{B})$, and expressions of the form $y_1^{a_1}\cdots y_m^{a_m}$ with integers $a_1,\ldots,a_m\geq 0$ \textit{cluster monomials}. Note that in this case we have $\mathcal{A}(\mathbf{x},\tilde{B})=\mathcal{A}(\mathbf{y},\tilde{C})$. 

Note that the frozen variables $x_{n+1},\ldots,x_m$ remain fixed under every mutation $\mu_k$ with $1\leq k\leq n$. Hence, they belong to every cluster of $\mathcal{A}(\mathbf{x},\tilde{B})$. They also admit an interpretation as \textit{coefficients} in the sense of Fomin-Zelevinsky \cite{FZ4}.

For every $1\leq j\leq n$ define a polynomial $f_j=\prod_{b_{ik} > 0} x_i^{b_{ik}} +\prod_{b_{ik} < 0} x_i^{-b_{ik}}$. Note that every $f_j$ is an element in the cluster algebra, but in this article we view it as an element $f_j\in K[x_1,\ldots,x_m]$. The exchange relations relating the initial seed and its neighboring seeds may be rewritten as $x_jx_j'=f_j$ for all $1\leq j\leq n$. Therefore, we call the polynomials $f_j$ the \textit{initial exchange polynomials}.

We also refer to an index $i\in\{1,2,\ldots,n\}$ as a \textit{mutable} index and to an index $j\in \{n+1,n+2,\ldots,m\}$ as a \textit{frozen} index. 

Define an unoriented simple graph $\Delta(\tilde{B})$ as follows: The vertex set is the set $\{1,2,\ldots ,m\}$ of all (mutable or frozen) indices and there is an edge between two vertices $i$ and $j$ with $i\geq j$ if and only if the entry $b_{ij}$ exists and is non-zero. We say that two (mutable or frozen) indices $i,j\in\{1,2,\ldots,m\}$ are \textit{adjacent} if they are adjacent in $\Delta(\tilde{B})$. Moreover, we say that $\tilde{B}$ is \textit{connected} if $\Delta(\tilde{B})$ is connected. Note that the connectedness of the exchange matrix is mutation invariant, i.e., if $\tilde{B}$ is connected, then $\mu_k(\tilde{B})$ is connected for all $1\leq k\leq n$. 

If, by chance, the principal part $B$ of the initial exchange matrix happens to be not only skew-symmetrizable, but skew-symmetric, then we may view $\tilde{B}$ as a signed adjacency matrix of a quiver $Q(\tilde{B})$ which $n$ mutable and $m-n$ frozen vertices. In this case, mutation of exchange matrices is the same as quiver mutation. In this case, we also write $\mathcal{A}(\mathbf{x},Q)$ for $\mathcal{A}(\mathbf{x},\tilde{B})$. Conversely, given a quiver $Q$, we denote by $B(Q)$ its signed adjacency matrix. Moreover, we usually denote the set of vertices of a quiver $Q$ by $Q_0$ and the set of arrows by $Q_1$. 

Finally, a mutable index $i\in\{1,2,\ldots,n\}$ is called a \textit{source} if $b_{ij}\geq 0$ for all $1\leq j\leq m$; it is called a \textit{sink} if $b_{ij}\leq 0$ for all $1\leq j\leq m$. Note that if the principal part $B$ is skew-symmetric, then the notions of sources and sinks agree with the corresponding notions for sources and sinks in the quiver $Q(\tilde{B})$. No index can be a source and a sink at the same time if $\tilde{B}$ is connected and $m\geq 2$.

\subsection{Main results on cluster algebras: The Laurent phenomenon and lower and upper bounds}
\label{BN2}
To continue our short survey, we now explain the main results on lower and upper bounds due to Berenstein-Fomin-Zelevinsky. The most striking result about cluster algebras is the \textit{Laurent phenomenon} according to which every cluster variable is not only a rational function but a Laurent polynomial in the variables of the initial cluster. Since every seed $(\mathbf{y},\tilde{C})$ of $\mathcal{A}(\mathbf{x},\tilde{B})$ can serve as an initial seed to construct the cluster algebra, the same statement is true for every cluster $\mathbf{y}$ of $\mathcal{A}(\mathbf{x},\tilde{B})$. To state the Laurent phenomenon more precisely let us put
\begin{align*}
&\mathcal{L}_{\mathbf{y}}=K[y_1^{\pm 1},\ldots,y_n^{\pm 1},y_{n+1}^{\pm 1},\ldots,y_m^{\pm 1}],\\
&\mathcal{L}_{\mathbf{y},\mathbb{Z}}=\mathbb{Z}[y_1^{\pm 1},\ldots,y_n^{\pm 1},y_{n+1}^{\pm 1},\ldots,y_m^{\pm 1}]
\end{align*}
for every cluster $\mathbf{y}$ of $\mathcal{A}(\mathbf{x},\tilde{B})$. The Laurent phenomenon asserts the following, see Fomin-Zelevinsky \cite[Theorem 3.1]{FZ}, \cite[Proposition 11.2]{FZ2}.
\begin{theorem}[Laurent phenomenon]
\label{LaurentPhenomenon}
Let $y$ be a cluster variable of $\mathcal{A}(\mathbf{x},\tilde{B})$. We have 
\begin{align*}
y\in  \bigcap_{(\mathbf{y},\tilde{C}) \sim (\mathbf{x},\tilde{B})} \mathcal{L}_{\mathbf{y},\mathbb{Z}}&&\textrm{and}&&
\mathcal{A}(\mathbf{x},\tilde{B}) \subseteq  \bigcap_{(\mathbf{y},\tilde{C}) \sim (\mathbf{x},\tilde{B})} \mathcal{L}_{\mathbf{y}}.
\end{align*}
\end{theorem} 
\noindent Following Berenstein-Fomin-Zelevinsky \cite[Section 1.3]{BFZ3} we introduce three $K$-algebras. They play a crucial role in the proof of the Laurent phenomenon. Put 
\begin{align*}
&\mathcal{L}(\mathbf{x},\tilde{B})=K[x_1,x_1',\ldots,x_n,x_n',x_{n+1}^{\pm 1},\ldots,x_m^{\pm 1}],\\
&\overline{\mathcal{A}}(\mathbf{x},\tilde{B})=\bigcap_{(\mathbf{y},\tilde{C}) \sim (\mathbf{x},\tilde{B})} \mathcal{L}_{\mathbf{y}},\quad
\mathcal{U}(\mathbf{x},\tilde{B})=\mathcal{L}_{\mathbf{x}}\cap\bigcap_{k=1}^n \mathcal{L}_{\mu_k(\mathbf{x})}.
\end{align*} 
Here, we use the notation $x_k'=\mu_k(\mathbf{x})_k$ for $1\leq k\leq n$. The $K$-algebras are called the \textit{lower bound}, the \textit{upper cluster algebra}, and the \textit{upper bound}. The choice of names reflects the fact that the inclusions $\mathcal{L}(\mathbf{x},\tilde{B})\subseteq \mathcal{A}(\mathbf{x},\tilde{B}) \subseteq \overline{\mathcal{A}}(\mathbf{x},\tilde{B}) \subseteq \mathcal{U}(\mathbf{x},\tilde{B})$ hold true. The first and the third inclusion are trivial and the second inclusion is a consequence of the Laurent phenomenon. 

Berenstein-Fomin-Zelevinsky \cite{BFZ3} also study under which conditions inclusions in the chain become equalities. We say that a seed $(\mathbf{y},\tilde{C})$ is \textit{coprime} if the initial exchange polynomials $f_k$ with $1\leq k\leq n$ are pairwise coprime elements in the ring $K[x_i\colon 1\leq i\leq m]$. Define a quiver $\Sigma(\tilde{C})$ with vertices $1,\ldots,n$ by drawing an arrow $i\to j$ between two vertices $i,j\in\{1,\ldots,n\}$ if and only if $c_{ij}>0$. We say that the seed $(\mathbf{y},\tilde{C})$ is \textit{acyclic} if the quiver  $\Sigma(\tilde{C})$ does not contain an oriented cycle. In this case, we also call the matrix $\tilde{C}$ also $\textit{acyclic}$. Additionally, we call the cluster algebra $\mathcal{A}(\mathbf{x},\tilde{B})$ \textit{acyclic} if at least one of its seeds is acyclic. Note that neither coprimality nor acyclicity is preserved under mutations: the coefficient-free cluster algebra associated with the quiver $1\rightarrow 2\rightarrow 3$ of type $A_3$ contains coprime, not coprime, acyclic, and not acyclic seeds.   

We summarize Berenstein-Fomin-Zelevinsky's results in the following theorem. References are Berenstein-Fomin-Zelevinsky \cite[Theorem 1.20, Proposition 1.8, Corollary 1.9, Theorem 1.18]{BFZ3}.

\begin{theorem}[Berenstein-Fomin-Zelevinsky]
\label{LowerUpperBound}
Let $K=\mathbb{Z}$. The following implications are true:
\begin{enumerate}
\item[(1)] If the seed $(\mathbf{x},\tilde{B})$ is acyclic, then the lower bound coincides with the cluster algebra, i.e., we have $\mathcal{L}(\mathbf{x},\tilde{B})=\mathcal{A}(\mathbf{x},\tilde{B})$.
\item[(2)] If $\operatorname{rank}(\tilde{B})=n$, then all seeds of $\mathcal{A}(\mathbf{x},\tilde{B})$ are coprime and the upper cluster algebra coincides with the upper bound, i.e., $\overline{\mathcal{A}}(\mathbf{x},\tilde{B})=\mathcal{U}(\mathbf{x},\tilde{B})$. 
\item[(3)] If the seed $(\mathbf{x},\tilde{B})$ is coprime and acylic, then all inclusion become equality, i.e., we have $\mathcal{L}(\mathbf{x},\tilde{B})=\mathcal{A}(\mathbf{x},\tilde{B})=\overline{\mathcal{A}}(\mathbf{x},\tilde{B})= \mathcal{U}(\mathbf{x},\tilde{B})$.
\end{enumerate}
\end{theorem}

Theorem \ref{LowerUpperBound} implies that a cluster algebra over $K=\mathbb{Z}$ is equal to its upper cluster algebra as soon as it has at least one coprime and acyclic seed. Moreover, the theorem yields ring theoretic properties for these cluster algebras. In particular, if the seed $(\mathbf{x},\tilde{B})$ is acyclic, then $\mathcal{A}(\mathbf{x},\tilde{B})$ is finitely generated and hence Noetherian. 

However, Berenstein-Fomin-Zelevinsky \cite[Proposition 1.26]{BFZ3} show that there exists a cluster algebras with $\mathcal{A}(\mathbf{x},\tilde{B})\neq\overline{\mathcal{A}}(\mathbf{x},\tilde{B})$. The example is the so-called \textit{Markov cluster algebra} with initial exchange matrix
\begin{align*}
\tilde{B}=\begin{pmatrix}0&2&-2\\-2&0&2\\2&-2&0\end{pmatrix}.
\end{align*}
It is known that the Markov cluster algebra is neither Noetherian and nor finitely generated, see Muller \cite[Proposition 11.3]{M}. More recently, Muller \cite[Theorem 12.5]{M2} has found more examples of cluster algebras which do not equal their lower bounds. More precisely, if $\mathcal{A}(\mathbf{x},\tilde{B})$ is a cluster associated with a triangulated marked surface with exactly one boundary component with exactly one marked point, then $\mathcal{A}(\mathbf{x},\tilde{B})\neq\overline{\mathcal{A}}(\mathbf{x},\tilde{B})$.

The second part of Theorem \ref{LowerUpperBound} remains true in the general case where $K$ is an arbitrary field, see Gei\ss-Leclerc-Schr\"oer \cite[Theorem 1.2]{GLS}. The same is true for the first part of Theorem \ref{LowerUpperBound}, as every cluster variable $y$ is an element in $\mathcal{L}_{\mathbf{x},\mathbb{Z}}\subseteq \mathcal{L}_{\mathbf{x}}$. 

Gei\ss-Leclerc-Schr\"oer \cite[Section 1.2]{GLS} consider a more general class of cluster algebras. The authors fix an integer $p\in \{n+1,n+2,\ldots,m\}$ and remove the inverses $x_{p+1}^{-1},x_{p+2}^{-1},\ldots,x_{m}^{-1}$ from the generating set $(\ref{GeneratingSet})$ of a cluster algebra. We do not treat those cluster algebras because the first part of Theorem \ref{LowerUpperBound}, with is crucial for our further investigations, does not remain true in the general case where $p<m$. To construct a counterexample, put $n=p=2$, $m=3$ and $K=\mathbb{Q}$. Let $x_1,x_2$ be two algebraically independent variables and $\mathbf{x}=(x_1,x_2)$. Put 
\begin{align*}
\tilde{B}=\begin{pmatrix}0&1\\-1&0\\1&-1\end{pmatrix}.
\end{align*}
Then $\mathcal{A}(\mathbf{x},\tilde{B})$ is a cluster algebra of rank 2 generated by five cluster variables. The initial seed is acyclic. The non-initial cluster variables are 
\begin{align*}
x_1'=\frac{x_2+x_3}{x_1}, &&x_2'=\frac{x_1+x_3}{x_2}, &&z=\frac{x_1+x_2+x_3}{x_1x_2}.
\end{align*} 
It follows that $\mathcal{L}(\mathbf{x},\tilde{B})=\mathbb{Q}[x_1,x_1',x_2,x_2',x_3]\subseteq \mathbb{Q}[x_1^{\pm 1},x_2^{\pm 1},x_3]$ is generated by homogeneous elements of nonnegative degrees $\operatorname{deg}(x_1)=\operatorname{deg}(x_2)=\operatorname{deg}(x_3)=1$ and $\operatorname{deg}(x_1')=\operatorname{deg}(x_2')=0$. Hence, the lower bound does not contain the cluster variable $z$ which is homogeneous of degree $-1$ and is not equal to the cluster algebra $\mathcal{A}(\mathbf{x},\tilde{B})$. However, we expect the statement to be true when $\tilde{B}$ is \textit{completely acyclic}, i.e., when the graph obtained from $\Sigma(\tilde{B})$ by adding the vertices $n+1,n+2,\ldots,m$ and arrows corresponding to non-zero entries in $\tilde{B}$ between mutable and frozen indices does not contain oriented cycles.

\subsection{The cluster algebra as a ring I: irreducible elements and unique factorization domains}
\label{BN3}
Let us recall some definitions from ring theory. For this purpose let $R$ be a commutative unital ring. We say that an element $r\in R$ is \textit{invertible} if there exists an $s\in R$ such that $rs=1$. The subset of invertible elements is denoted by $R^{\times}\subseteq R$. Two elements $r,s\in R$ are called \textit{associated} if there exists an $t\in R^{\times}$ such that $r=st$. 

We say that $r\in R\backslash R^{\times}$ is \textit{irreducible} if the following implication is true: If $r=r_1r_2$ for some $r_1,r_2\in R$, then $r_1$ is invertible or $r_2$ is invertible. We say that an element $r\in R$ divides an element $s\in R$ if there exists an $t\in R$ such that $s=rt$. In this case we write $r\vert s$. Moreover, an element $r\in R$ is called \textit{prime} if the following implication is true: If $s,t\in R$ are elements such that $r\vert st$, then $r\vert s$ or $r\vert t$. It is a standard argument to show that every prime element is irreducible.

In the context of cluster cluster the case $m=1$ is special. Note that in this case $\mathcal{A}(\mathbf{x},\tilde{B})=K[x_1,\frac2{x_1}]$. As a Laurent polynomial ring it is a unique factorization domain for all fields $K$; it is also well-known to be a unique factorization domain for $K=\mathbb{Z}$. For $m>1$ the following theorem holds. 

\begin{theorem}[Gei\ss-Leclerc-Schr\"oer]
\label{GLS}Let $B$ be an integer $m\times n$ matrix with $m\geq2$. Let $\mathcal{A}(\mathbf{x},\tilde{B})$ be the cluster algebra associated with the seed $(\mathbf{x},\tilde{B})$. 
\begin{enumerate}
\item[(1)] The set of invertible elements is 
\begin{align*}
\mathcal{A}(\mathbf{x},\tilde{B})^{\times}=\{\lambda x_{n+1}^{a_n+1}\cdots x_m^{a_m}\colon \lambda\in K^{\times}, a_{n+1},\ldots,a_p\in\mathbb{Z}\}.
\end{align*}
\item[(2)] Every cluster variable of $\mathcal{A}(\mathbf{x},\tilde{B})$ is irreducible.
\end{enumerate}
\end{theorem}

Especially, $\mathcal{A}(\mathbf{x},\tilde{B})^{\times}/K^{\times}$ is a free abelian group of rank $m-n$. If $K=\overline{K}$ is an algebraically closed field and the seed $(\mathbf{x},\tilde{B})$ is acyclic (so that $\mathcal{A}(\mathbf{x},\tilde{B})$ is a finitely generated $K$-algebra), then this fact illustrates Samuel's theorem.  

The reference for Theorem \ref{GLS} is Gei\ss-Leclerc-Schr\"oer \cite[Theorem 1.3]{GLS}. In the same article the authors study further ring theoretic properties of cluster algebras. In particular, they investigate whether a given cluster algebra is a \textit{unique factorization domain}. Recall that a ring $R$ is a unique factorization domain if two statements hold: First, every element $r\in R$ can be written as a product $r=ur_1\cdots r_k$ for some $k\geq 0$, some irreducible elements $r_i\in R$, and some unit $u\in R^{\times}$. Second, the above factorization of $r$ is unique in the sense that if $r=vs_1\cdots s_l$ for some $l\geq 0$, some irreducible elements $s_i\in R$, and some unit $v\in R^{\times}$, then $k=l$ and there exists a bijection $\varphi\colon\{1,\ldots,k\}\to \{1,\ldots,k\}$ such that $r_i$ and $s_{\varphi(i)}$ are associated for all $1\leq i\leq k$. Other authors also use the term \textit{factorial ring} for a unique factorization domain. It is well-known that in a unique factorization domain every irreducible element is prime.

Not every cluster algebra is a unique factorization domain. Zelevinsky (cf. \cite[Section 1.9]{GLS}) asks whether $\mathcal{A}(\mathbf{x},\tilde{B})$ is a unique factorization domain if $\operatorname{rank}(\tilde{B})=n$. We answer the question in the negative.

\begin{ex} 
\label{NotUFD}
Let $\mathcal{A}(\mathbf{x},\tilde{B})$ be the coefficient-free cluster algebra associated with the Kronecker quiver over the complex numbers, i.e., $K=\mathbb{C}$, $m=n=2$, $\mathbf{x}=(x_1,x_2)$, and 
\begin{align*}
\tilde{B}=B=\begin{pmatrix}0&2\\-2&0\end{pmatrix}.
\end{align*}
Then the cluster variable $x_1$ is not a prime element. Since it irreducible it follows that $\mathcal{A}(\mathbf{x},\tilde{B})$ is not a unique factorization domain, although $\operatorname{rank}(\tilde{B})=2$. 
\end{ex}

\begin{proof}
The rank of $\tilde{B}$ clearly equals $2$. As above we put $x_k'=\mu_k(\mathbf{x})_k$ for $k\in\{1,2\}$. We have $x_1x_1'=1+x_2^2=(1+ix_2)(1-ix_2)$. Note that $1+ix_2,1-ix_2\in \mathcal{A}(\mathbf{x},\tilde{B})$. Suppose that $x_1$ is prime. Then we have $x_1\vert1+ix_2$ or $x_1\vert 1-ix_2$. Without loss of generality we may assume that $x_1\vert 1+ix_2$ which means that there exist some $s\in \mathcal{A}(\mathbf{x},\tilde{B})\subseteq \mathbb{C}[x_1^{\pm 1},x_2^{\pm 1}]$ such that $sx_1=1+ix_2$. Let us write $s=s_1+is_2$ for some $s_1,s_2\in\mathbb{R}[x_1^{\pm 1},x_2^{\pm 1}]$ and put $\overline{s}=s_1-is_2$. Note that $\overline{s}\in \mathcal{A}(\mathbf{x},\tilde{B})$ because the the generators $x_1,x_1',x_2,x_2' $ of $\mathcal{A}(\mathbf{x},\tilde{B})$ lie in $\mathbb{R}[x_1^{\pm 1},x_2^{\pm 1}]$. It follows that $\overline{s}x_1=1-ix_2$. Thus, we have $(s+\overline{s})x_1=2\in  \mathcal{A}(\mathbf{x},\tilde{B})^{\times}$ which yields $x_1\in  \mathcal{A}(\mathbf{x},\tilde{B})^{\times}$ which is impossible. 
\end{proof}

Gei\ss-Leclerc-Schr\"oer \cite[Section 6.1]{GLS} provide two general methods to construct cluster algebras which are not unique factorization domains. 
 
\begin{prop}[Gei\ss-Leclerc-Schr\"oer] 
\label{ReducibleImpliesNonFactorial}
If there is some $i\in\{1,\ldots,n\}$ such that $f_i$ is reducible in $K[x_i\colon 1\leq i\leq m]$, then $\mathcal{A}(\mathbf{x},\tilde{B})$ is not a unique factorization domain.
\end{prop}
\begin{prop}[Gei\ss-Leclerc-Schr\"oer] 
\label{CoincidenceImpliesNonFactorial}
If there are two distinct indices $i,j\in\{1,\ldots,n\}$ such that $f_i=f_j$, then $\mathcal{A}(\mathbf{x},\tilde{B})$ is not a unique factorization domain.
\end{prop}

\noindent Proposition \ref{ReducibleImpliesNonFactorial} is a natural generalization of Example \ref{NotUFD}. The main idea for the proof of Proposition \ref{CoincidenceImpliesNonFactorial} is to consider the factorization $x_ix_i'=x_jx'_j$.

\subsection{Skew symmetric cluster algebras of finite type and hypersurfaces}
\label{BN4}
In this subsection we shall illustrate the above results by considering a remarkable class of coefficient-free skew-symmetric cluster algebra $\mathcal{A}(\mathbf{x},Q)$, namely those attached to orientations of Dynkin diagrams of type $A, D$ or $E$. This class is remarkable in two ways: A result by Fomin-Zelevinsky \cite[Theorem 1.4]{FZ2} asserts that $\mathcal{A}(\mathbf{x},Q)$ has only finitely many cluster variables if and only if it has a seed $(\mathbf{x}',Q')$ such that $Q'$ is an orientation of a Dynkin diagram of type $A, D$ or $E$. In this case we say that $\mathcal{A}(\mathbf{x},Q)$ is of \textit{finite type}. Thus, we obtain a Cartan-Killing classification for cluster algebras of finite type. The number of non-initial cluster variables equals the number of positive roots in the corresponding complex simple Lie algebra. Now let $\Delta$ be a Dynkin diagram of type $A, D$ or $E$. Another result by Fomin-Zelevinsky \cite[Theorem 1.5]{FZ2} asserts that the coefficient-free cluster algebras $\mathcal{A}(\mathbf{x},\overrightarrow{\Delta})$ are isomorphic for all orientations $\overrightarrow{\Delta}$ of $\Delta$. We also write $\mathcal{A}(\mathbf{x},\Delta)$ for $\mathcal{A}(\mathbf{x},\overrightarrow{\Delta})$.

The next example shows that Gei\ss-Leclerc-Schr\"oer's first criterion can be applied to $\mathcal{A}(\mathbf{x},A_3)$ and to $\mathcal{A}(\mathbf{x},D_n)$ for $n\geq 4$. 

\begin{ex} 
\label{Dynkin}
Put 
\begin{align*}
\tilde{B}=B=\begin{pmatrix}0&1&0\\-1&0&1\\0&-1&0\end{pmatrix}.
\end{align*}
Then $\mathcal{A}(\mathbf{x},\tilde{B})$ is a cluster algebra of type $A_3$ with initial seed $(\mathbf{x},\tilde{B})$. We have $f_1=f_3=1+x_2\in K[x_1,x_2,x_3]$, so by Proposition \ref{CoincidenceImpliesNonFactorial} $\mathcal{A}(\mathbf{x},A_3)$ cannot be a unique factorization domain. The same is true for the cluster of type $D_n (n\geq 4)$: Let $Q=\overrightarrow{D_n}$ be any orientation of the Dynkin diagram of type $D_n$, and $i,j\in Q_0$ be different neighbors of the vertex $l$ of degree $3$. Then $f_i=f_j=1+x_l$, so $\mathcal{A}(\mathbf{x},D_n)$ cannot be a unique factorization domain.   
\end{ex}

Propositions \ref{CoincidenceImpliesNonFactorial} and \ref{ReducibleImpliesNonFactorial} do not apply to  cluster algebras of the form $\mathcal{A}(\mathbf{x},Q)$ for a quiver $Q$ of type $A_n$ with $n\neq 1,3$ or $E_n$ with $n=6,7,8$. We will see later in Section \ref{Applications} that such a cluster algebra $\mathcal{A}(\mathbf{x},Q)$ actually is a unique factorization domain. 

For the moment we just make a remark on size of generating sets of the cluster algebra $\mathcal{A}(\mathbf{x},A_n)$. We have reduced the cardinality of a generating set of the cluster algebra twice -- from a priori infinitely many cluster variables to $n+\frac{n(n+1)}{2}$ cluster variables by finite type classification and then to $2n$ by acyclicity. The following argument shows that we can do better. In fact,  $\mathcal{A}(\mathbf{x},A_n)$ is generated by $n+1$ elements. Let $Q$ be the so-called linear orientation of the Dynkin diagram $A_n$, i.e., $Q_0=\{1,2,\ldots,n\}$ and $Q_1=\{i\rightarrow i+1 \colon 1\leq i\leq n-1\}$. Assume that $n\geq 2$. The cluster algebra is generated by $x_1,x_2,\ldots,x_n$ and $x_1',x_2',\ldots,x_n'$. The system 
\begin{align*}
x_1x_1&'=1+x_2\\
x_2x_2'&=x_1+x_3\\
x_3x'_3&=x_2+x_4\\
&\vdots\\
x_nx'_n&=1+x_{n-1}
\end{align*}
of relations enables us to successively replace $x_2,x_3,\ldots,x_n$ by a polynomial expression in the remaining variables. We see that $\mathcal{A}(\mathbf{x},Q)=K[x_1,x'_i\colon 1\leq i\leq n]/(P_n)$ for some polynomial $P_n\in K[x_1,x'_i\colon 1\leq i\leq n]$. For example, we have 
\begin{align*}
&P_2=x_1x_1'x'_2-x_1-x_2'-1,\\
&P_3=x_1x'_1x'_2x'_3-x'_2x'_3-x_1x'_3-x_1x'_1.
\end{align*}   
There is a recursion formula which enables us to compute all further polynomials. We have $P_n=x'_nP_{n-1}+x'_n-P_{n-2}-2$ for $n\geq 4$ which can be proved inductively. Note that every $P_n$ is an irreducible polynomial in $K[x_1,x'_i\colon 1\leq i\leq n]$, because $\mathcal{A}(\mathbf{x},A_n)$ is an integral domain.

\section{A criterion for acyclic cluster algebras to admit unique factorization}
\label{Description}

\subsection{Assumptions on the cluster algebra}
\label{Assumptions}
Let $m\geq n\geq 1$ be integers with $m\geq 2$ and let $\mathcal{A}(\mathbf{x},\tilde{B})$ be the cluster algebra associated with an initial seed $(\mathbf{x},\tilde{B})$; here $\tilde{B}$ is as usual an $m\times n$ integer matrix with a skew-symmetrizable principal part $B$ and $\mathbf{x}=(x_1,x_2,\ldots, x_m)$ is an initial cluster consisting of $m$ variables that are algebraically independent over $K$. Furthermore, in the rest of Section \ref{Description} we assume that the following conditions hold:
\begin{enumerate}
\item The matrix $\tilde{B}$ is connected.
\item The initial seed $(\mathbf{x},\tilde{B})$ is acyclic.
\item The polynomials $f_i$ with $1\leq i\leq n$ are pairwise coprime.
\item Every polynomials $f_i$ with $1\leq i\leq n$ is irreducible.
\end{enumerate}
Hence, the seed $(\mathbf{x},\tilde{B})$ is coprime and by Theorem \ref{LowerUpperBound} the cluster algebra $\mathcal{A}(\mathbf{x},\tilde{B})$ is equal to its lower and its upper bound. For ring theoretic studies the connectedness of $\tilde{B}$ is not strong assumption, because the cluster algebra associated with a seed with a disconnected exchange matrix is naturally isomorphic to a product of cluster algebras associated with seeds with connected initial exchange matrices. Together with the assumption $m\geq 2$ it ensures that no initial exchange polynomial is constant. 

\subsection{The definition of the ideals and their algebraic properties}
\label{Ideals}

The next definition introduces ideals that will be crucial for our ring theoretic study of the cluster algebra $\mathcal{A}(\mathbf{x},\tilde{B})$.

\begin{defn} For every $i\in\{1,2,\ldots,n\}$ define an ideal 
\begin{align*}
I_i=(x_i,f_i)\subseteq K[x_1,x_2,\ldots,x_m].
\end{align*}
\end{defn}
Before describing the significance of the ideals in the context of cluster algebras let us present some of their properties. Proposition \ref{OneIdeal} and Lemma \ref{Cancelling} will be especially helpful for all further discussions. 

Let us introduce the abbreviation $R=K[x_i\colon 1\leq i\leq m]$ for the polynomial ring. Moreover, for $1\leq i\leq n$ we abbreviate $K[x_1,\ldots,\widehat{x_i},\ldots,x_m]$ by $R_i$. Note that $f_i\in R_i$ since the diagonal entry $b_{ii}$ of $B$ is zero for all $1\leq i\leq n$.

\begin{prop}
\label{OneIdeal}
Let $1\leq i\leq n$ and let $a_i\in\mathbb{N}$. Let $P\in R$ be a polynomial which we write as $P=\sum_{k=0}^{\infty}P_kx_i^k$ for some polynomials $P_k\in R_i$. Then $P\in I_i^{a_i}$ holds if and only if the polynomial $f_i^{a_i-k}\in R_i$ divides the polynomial $P_k\in R_i$ for all $0\leq k\leq a_i$. 
\end{prop}

\begin{proof} Assume that $P\in I_i^{a_i}$. By definition, $P=\sum_{r=0}^{a_i}A_rx_i^rf_i^{a_i-r}$ for some polynomials $A_r\in R$. Every $A_r$ can be written as a finite sum $A_r=\sum_{s= 0}^{\infty}A_{rs}x_i^s$ with polynomials $A_{rs}\in R_i$. Hence $P=\sum_{r,s}A_{rs}x_i^{r+s}f_i^{a_i-r}$ so that $P_k=\sum_{r+s=k}A_{rs}f_i^{a_i-r}$ is divisible by $f_i^{a_i-k}$ for all $0\leq k\leq a_i$. The reverse direction follows immediately.
\end{proof}

In particular, we have $x_i\notin I_j$ for $i\neq j$. For a polynomial such as $P\in R$ in Proposition \ref{OneIdeal} it is usual to denote part $P_k$ of degree $k$ with respect to $x_i$ by $[x_i^k]P$.

It follows that for every non-zero polynomial $P\in R$ and every $1\leq i\leq n$ there is a largest natural number $a_i\in\mathbb{N}$ such that $P\in I_i^{a_i}$. We define $m_i(P)$ to be the unique natural number such that $P\in I_i^{m_i(P)}\backslash I_{i}^{m_i(P)+1}$. Moreover, we define a momomial $$M(P)=\prod_{i=1}^nx_i^{m_i(P)}\in R.$$

\begin{rem}
\label{ViceVersa}
If, by chance, $f_i$ happens to have the form $f_i=x_l+M_i$ for some mutable or frozen vertex $l\neq i$, then $R=K[x_1,\ldots,\widehat{x_l},\ldots,x_m,f_i]$ and we may write every polynomial $P\in R$ in the form $P=\sum_{r=0}^{\infty} A_rf_i^r$ with $A_r\in K[x_1,\ldots,\widehat{x_l},\ldots,x_m]$. In this case, by the same argument as above, we have $P\in I_i^{a_i}$ for some $a_i\geq 0$ if and only if $x_i^{a_i-r}\vert P_r$ for all $0\leq r\leq a_i$.  
\end{rem}

\begin{prop}
\label{InitialPolyNotInOtherIdeal}
For all mutable indices $i\neq j$ the initial exchange polynomial $f_i$ is not an element in the ideal $I_j$.
\end{prop}

\begin{proof} We distinguish two cases. If $i$ and $j$ are connected, then $f_i=M_i+M_i'$ is a sum of two mononials one of which is divisible by $x_j$. It follows that $[x_j^0]f_i$ of $f_i$ is a monomial, and hence not divisible by $f_j$. By Proposition \ref{OneIdeal} we have $f_i\notin I_j$. If $i$ and $j$ are not connected, then $[x_j^0]f_i=f_i$. The coprimality of $f_i$ and $f_j$ implies that $f_i$ is not divisible by $f_j$. By the same argument as above we have $f_i\notin I_j$.  
\end{proof}

\begin{lemma}
\label{Cancelling}
Let $1\leq i\leq n$ and let $a_i\geq 1$ be a natural number. If $P,Q \in R$ are polynomials satisfying $PQ\in I_i^{a_i}$, then there exists some $0\leq b_i\leq a_i$ such that $P\in I_i^{b_i}$ and $Q\in I_i^{a_i-b_i}$.
\end{lemma}

\begin{proof} It follows from Proposition \ref{OneIdeal} that there exists a largest natural number $k$ such that $P\in I_i^k$. Similarly, let $l$ be the largest natural number such that $Q\in I_i^l$. We want to show that $k+l\geq a_i$. Assume on the contrary that $k+l+1\leq a_i$.

There are polynomials $A_r,B_t\in R$ for $0\leq r\leq k$ and $0\leq t\leq l$ such that 
\begin{align*}P=\sum_{r=0}^kA_rf_i^{k-r}x_i^r, &&Q=\sum_{t=0}^lB_tf_i^{l-t}x_i^t. 
\end{align*} 
Every $A_r$ can be written as a finite sum $A_r=\sum_{s=0}^{\infty}A_{rs}x_i^s$ with $A_{rs}\in R_i$. Similarly, for every polynomials $B_t$ there polynomials $B_{tu}\in R_i$ with $u\in\mathbb{N}$ such that $B_t$ can be written as $B_t=\sum_{u= 0}^{\infty}B_{tu}x_i^u$. In this notation the product is equal to
\begin{align*}PQ=\sum_{0\leq r\leq k,0\leq t\leq l\atop s,u\geq 0}A_{rs}B_{tu}f_i^{k+l-r-t}x_i^{r+s+t+u}.
\end{align*}
We conclude that 
\begin{align*}
PQ\equiv \sum_{0\leq r\leq k\atop0\leq t\leq l}A_{r0}B_{t0}f_i^{k+l-r-t}x_i^{r+t} &&(\textrm{mod} \ I_i^{k+l+1}). 
\end{align*}
Note that $A_{r0}B_{t0}f_i^{k+l-r-t}\in R_i$ for all $r,t\geq 0$ and that $PQ\in I_i^{a_i}\subseteq I_i^{k+l+1}$. By Proposition \ref{OneIdeal} we have $f_i \vert \sum_{r+t=N}A_{r0}B_{t0}$ for all $0\leq N\leq k+l$. Using the irreducibility of $f_i$ over $K$ it is a standard argument to show that either we have $f_i \vert A_{r0}$ for all $0\leq r\leq k$, in which case we get $P\in I_i^{k+1}$, or we have $f_i \vert B_{t0}$ for all $0\leq t\leq l$, in which case we get $Q\in I_i^{l+1}$. Contradiction.
\end{proof}

Besides being useful for the study of cluster algebras the ideals satisfy formidable algebraic properties. 

Recall from commutative algebra that an ideal $I\subseteq R=K[x_1,x_2,\ldots,x_m]$ is called a \textit{prime ideal} if $I\neq R$ and for all elements $P,Q\in R$ the following implication holds: If $PQ\in I$, then $P\in I$ or $Q\in I$. The \textit{radical ideal} of an ideal $I\subseteq R$ is $\sqrt{I}=\{r\in R\colon r^k\in R \textrm{ for some } k\in \mathbb{N}\}$. An ideal $I\subseteq R$ is called a \textit{primary ideal} if for all $P,Q\in R$ the following implication holds: If $PQ\in I$, then $P\in I$ or $Q\in\sqrt{I}$. In this case, the radical ideal $\sqrt{I}$ is a prime ideal and $I$ is also called $\sqrt{I}$-primary.

\begin{lemma}
The ideal $I_i$ is a prime ideal for all $1\leq i\leq n$.
\end{lemma}

\begin{proof}
Let $1\leq i\leq n$. Suppose that $P,Q\in R$ are polynomials with $PQ\in I_i$. Put $a_i=1$. Lemma \ref{Cancelling} yields $P\in I_i$ or $Q\in I_i$.
\end{proof}

\begin{lemma}
\label{primary}
The ideal $I_i^{a_i}$ is a primary ideal for all $1\leq i\leq n$ and all natural numbers $a_i\geq 1$. The radical ideal of $I_i^{a_i}$ is equal to $I_i$.
\end{lemma}

\begin{proof}
Let $1\leq i\leq n$ and let $a_i\geq 1$ be a natural number. We start with the second statement $\sqrt{I_i^{a_i}}=I_i$. Suppose that $P\in I_i$. Then $P^{a_i}\in I_i^{a_i}$ and hence $P\in \sqrt{I_i^{a_i}}$. Conversely, suppose that $P\in \sqrt{I_i^{a_i}}$. Then there exists a natural number $k$ with $P^k\in I_i^{a_i}$. We prove by induction on $k\geq 1$ that $P^k\in I_i^{a_i}$ implies that $P\in I_i^l$ for some natural number $l \geq 1$. The claim implies that $P\in I_i$ since $I_i^l\subseteq I_i$. The case $k=1$ is trivial. Suppose that $k\geq 2$. Put $Q=P^{k-1}$. Lemma~\ref{Cancelling} implies that there exists $0\leq b_i\leq a_i$ such that $P\in I_i^{b_i}$ and $P^{k-1}\in I_i^{a_i-b_i}$. Note that $b_i\geq 1$ or $a_i-b_i\geq 1$. If $b_i\geq 1$, then we are done. Otherwise, we have $P^{k-1}\in I_i^{a_i}$ and we conclude by induction hypothesis.

For the first statement suppose that $P,Q\in R$ are polynomials with $PQ\in I_i^{a_i}$. By Lemma~\ref{Cancelling} we have $P\in I_i^{b_i}$ and $Q\in I_i^{a_i-b_i}$ for some $0\leq b_i\leq a_i$. If $b_i=a_i$, then $P \in I_i^{a_i}$. Otherwise, $a_i-b_i\geq 1$ and $Q\in I_i^{a_i-b_i}\subseteq I_i=\sqrt{I_i^{a_i}}$.
\end{proof}

Recall that two ideals $I,J\in R$ are called \textit{coprime} if $I+J=(1)=R$. In this case, a proposition from commutative algebra asserts that $I^a+J^b=(1)$ for all $a,b\geq 1$, i.e., the powers $I^a$ and $J^b$ are again coprime. Another well-known proposition asserts that $I\cap J=IJ$ for coprime ideals $I,J\subseteq R$. 

\begin{prop}
\label{coprime}
Suppose that $i\in\{1,2,\ldots,n\}$ is either a sink or a source and that the index $j\in\{1,2,\ldots,n\}$ is adjacent with $i$. Then the ideals $I_i$ and $I_j$ are coprime. 
\end{prop}

\begin{proof} Note that $f_i=1+M_i$ for some monomial $M_i\in R$. The adjacency of $i$ and $j$ implies $x_j\vert M_i$, so that $M_i\in I_j$. Hence $1=f_i-M_i\in I_i+I_j$, so that $I_i$ and $I_j$ are coprime.    
\end{proof}

\subsection{A description of the cluster algebra}
\label{des}
In this subsection we study the cluster algebra $\mathcal{A}(\mathbf{x},\tilde{B})$ from a ring theoretical point of view. The Laurent phenomenon (see Theorem \ref{LaurentPhenomenon}) and the equality of the cluster algebra with its lower bound (see Theorem \ref{LowerUpperBound}) imply that $\mathcal{A}(\mathbf{x},\tilde{B})$ may be described as a union: 

\begin{rem}
\label{Union}
We have
\begin{align*}
\mathcal{A}(\mathbf{x},\tilde{B})=\bigcup_{\mathbf{a}\in\mathbb{N}^n} \left\{ \frac{\lambda P}{x_1^{a_1}\cdot x_2^{a_2}\cdots x_n^{a_n}} \colon P\in I_1^{a_1}\cdot I_2^{a_2}\cdots I_n^{a_n},\lambda \in \mathcal{A}(\mathbf{x},\tilde{B})^{\times} \right\}.
\end{align*}
\end{rem} 
Let us introduce the shorthand multi-index notation $I^{\mathbf{a}}=I_1^{a_1}I_2^{a_2}\cdots I_n^{a_n}$ for a sequence $\mathbf{a}=(a_1,a_2,\ldots,a_n)\in \mathbb{N}^n$. Similarly, write $x^{\mathbf{a}}=x_1^{a_1}x_2^{a_2}\cdots x_n^{a_n}$ and $f^{\mathbf{a}}=f_1^{a_1}f_2^{a_2}\cdots f_n^{a_n}$ for a sequence $\mathbf{a}$ as above. Moreover, we denote by $\geq$ the \textit{product order} on $\mathbb{N}$, i.e., we write $\mathbf{a}\geq\mathbf{b}$ whenever two sequences $\mathbf{a},\mathbf{b}\in \mathbb{N}^n$ satisfy $a_i\geq b_i$ for all $1\leq i\leq n$.

\subsection{A conjectured primary decomposition}
\label{conjecture}
Recall that if $I,J\subset R$ are two ideals, then their product $I\cdot J$ is a subset of their intersection $I\cap J$, i.e., $I\cdot J\subseteq I\cap J$, but the sets are not equal in general. We conjecture the following.

\begin{conj}
\label{IdealProperties}
For all $\mathbf{a}\in\mathbb{N}^n$ we have $I^{\mathbf{a}}=I_1^{a_1}\cap I_2^{a_2}\cap\ldots\cap I_n^{a_n}$.
\end{conj}

By Lemma \ref{primary} the ideals $I_i^{a_i}$ are primary ideals for all $a_i\in\mathbb{N}$, so that the right hand side would be \textit{primary decomposition} of the ideal $I^{\mathbf{a}}$ for all $\mathbf{a}\in\mathbb{N}^n$. As we will see in the next subsection, the above conjecture implies that the cluster algebra $\mathcal{A}(\mathbf{x},\tilde{B})$ is a unique factorization domain. 

\begin{rem} 
\label{Union2}
If Conjecture \ref{IdealProperties} is true, then the union in Remark \ref{Union} becomes
\begin{align*}
\mathcal{A}(\mathbf{x},\tilde{B})=&\bigcup_{\mathbf{a}\in\mathbb{N}^n} \left\{\frac{\lambda P}{x_1^{a_1}\cdot x_2^{a_2}\cdots x_n^{a_n}}\colon P\in I_1^{a_1}\cap I_2^{a_2} \cap \ldots \cap I_n^{a_n},\lambda \in \mathcal{A}(\mathbf{x},\tilde{B})^{\times}\right\}.
\end{align*}
\end{rem}

\begin{rem}
Let $1\leq i\leq n$. Assume that for the sequence $\mathbf{a}\in\mathbb{N}^n$ we have $a_i\neq 0$ and that $I_1^{a_1}\cap I_2^{a_2}\cap\ldots\cap I_n^{a_n}$ contains some polynomial $x_iP$ which is divisible by $x_i$. We have $x_i\in I_i\backslash I_i^2$ and $x_i\notin I_j$ for $j\neq i$. Therefore, by Lemma~\ref{Cancelling} we have $P\in I_1^{a'_1}\cap I_2^{a_2'}\cap\ldots \cap I_n^{a'_n}$ for the sequence $\mathbf{a}'\in\mathbb{N}^n$ with $a_i'=a_i-1$ and $a'_j=a_j$ for $j\neq i$. In other words, the description of cluster algebra elements in the previous remark behaves well with respect to canceling a common factor $x_i$ in both the numerator and the denominator. 
\end{rem}

\begin{rem}
For all $\mathbf{a}\in\mathbb{N}^n$ define a set 
\begin{align*}S(\mathbf{a})=&\left\{P\in I_1^{a_1}\cap I_2^{a_2} \cap \ldots \cap I_n^{a_n} \colon x_i \nmid P \ {\rm if} \ 1\leq i\leq n, a_i\neq 0\right\}
\end{align*}
Especially, we have $R=S(0)$. If $I^{\mathbf{a}}=I_1^{a_1}\cap I_2^{a_2}\cap\ldots\cap I_n^{a_n}$ holds for all $\mathbf{a}\in\mathbb{N}^n$, then application of Lemma \ref{Cancelling} with $Q=x_i$ yields a decomposition 
\begin{align*}
\mathcal{A}(\mathbf{x},\tilde{B})=\bigcup_{\mathbf{a}\in\mathbb{N}^n}\frac{ \mathcal{A}(\mathbf{x},\tilde{B})^{\times}\cdot S(\mathbf{a})}{x_1^{a_1}\cdot x_2^{a_2}\cdots x_n^{a_n}}.
\end{align*}
\end{rem}

The following example shows that we cannot expect the conjecture to hold in the general case where $\tilde{B}$ is not acyclic.

\begin{ex} Put 
\begin{align*}
\tilde{B}=B=\begin{pmatrix}0&1&-1\\-1&0&1\\1&-1&0\end{pmatrix}.
\end{align*}
Then $\mathcal{A}(\mathbf{x},\tilde{B})$ is a cluster algebra of type $A_3$ with a cyclic seed $(\mathbf{x},\tilde{B})$. We have $I_1=(x_1,x_2+x_3)$, $I_2=(x_2,x_3+x_1)$, and $I_3=(x_3,x_1+x_2)$ and so $x_1+x_2+x_3\in I_1\cap I_2\cap I_3$, but $x_1+x_2+x_2\notin I_1I_2I_3$ as all elements in $I_1I_2I_3$ are linear combinations of monomials whose degree is at least 3. 
\end{ex}

\subsection{The criterion} In this subsection we give a sufficient condition for $\mathcal{A}(\mathbf{x},\tilde{B})$ to be a unique factorization domain. 
\label{Crit}

\begin{theorem}
\label{Criterion}
If $I^{\mathbf{a}}=I_1^{a_1}\cap I_2^{a_2}\cap\ldots\cap I_n^{a_n}$ holds for all $\mathbf{a}\in\mathbb{N}^n$, then $\mathcal{A}(\mathbf{x},\tilde{B})$ is a unique factorization domain. Moreover, the set irreducible elements in $\mathcal{A}(\mathbf{x},\tilde{B})$ is 
\begin{align*}
&\left(\left\{\lambda x_i\colon 1\leq i\leq n, \lambda\in \mathcal{A}(\mathbf{x},\tilde{B})^{\times} \right\} \right.\\
&\hspace{2cm}\left.\cup\left\{\frac{\lambda P}{M(P)}\colon P\in R \ {\rm irreducible},  \lambda\in \mathcal{A}(\mathbf{x},\tilde{B})^{\times} \right\} \right)\backslash \mathcal{A}(\mathbf{x},\tilde{B})^{\times}.
\end{align*}
\end{theorem}

\begin{proof} Let $\frac{\lambda P}{\mathbf{x}^{\mathbf{a}}}$ with $\mathbf{a}\in\mathbb{N}^n, P\in S(\mathbf{a})$ and $\lambda \in \mathcal{A}(\mathbf{x},\tilde{B})^{\times}$ be an element in the cluster algebra $\mathcal{A}(\mathbf{x},\tilde{B})$. Without loss of generality we may assume that $x_i\nmid P$ for $i\in\{n+1,n+2,\ldots,m\}$. Suppose that $P=FG$ is reducible for some non-constant polynomials $F,G \in R\backslash \mathcal{A}(\mathbf{x},\tilde{B})^{\times}$. We distinguish two cases. If $P=x_iQ$ is divisible by $x_i$ for some $1\leq i\leq n$, then $a_i=0$, $Q\in I^{\mathbf{a}}$ by Lemma \ref{Cancelling} and
\begin{align*}
\frac{\lambda P}{\mathbf{x}^{\mathbf{a}}}=\frac{\lambda Q}{\mathbf{x}^{\mathbf{a}}} \cdot x_i.\end{align*}
is the product of two cluster algebra elements. This element cannot be irreducible unless $\frac{\lambda Q}{\mathbf{x}^{\mathbf{a}}}\in \mathcal{A}(\mathbf{x},\tilde{B})^{\times}$. If $P$ is not divisible by any $x_i$, then by Lemma \ref{Cancelling} there exist sequences $\mathbf{b},\mathbf{c}\in\mathbb{N}^n$ of natural numbers  with $\mathbf{b}+\mathbf{c}=\mathbf{a}$ such that $F\in I^{\mathbf{b}}$ and $G\in I^{\mathbf{c}}$. Hence,  
\begin{align*}
\frac{\lambda P}{\mathbf{x}^{\mathbf{a}}}=\frac{\lambda F}{\mathbf{x}^{\mathbf{b}}} \cdot \frac{G}{\mathbf{x}^{\mathbf{c}}}
\end{align*}
is the product of two cluster algebra elements none of which is invertible in $\mathcal{A}(\mathbf{x},\tilde{B})$. In both cases $\frac{\lambda P}{\mathbf{x}^{\mathbf{a}}}$ cannot be irreducible in $\mathcal{A}(\mathbf{x},\tilde{B})$. Now assume that $P\in R$ is an irreducible polynomial. If $P\in S_{\mathbf{a'}}$ for some sequence $\mathbf{a'}\geq \mathbf{a}$, then 
\begin{align*}
\frac{\lambda P}{\mathbf{x}^{\mathbf{a}}}=\frac{\lambda P}{\mathbf{x}^{\mathbf{a}'}} \cdot \mathbf{x}^{\mathbf{a}'-\mathbf{a}}
\end{align*}
is the product of two element none of which is invertible in $\mathcal{A}(\mathbf{x},\tilde{B})$. This shows that every irreducible cluster algebra element, which is not associated with an initial cluster variable, has the form $\frac{\lambda P}{M(P)}$ for some irreducible polynomial $P\in R$ not divisible by any $x_i$ $(1\leq i\leq m)$ and some $\lambda\in \mathcal{A}(\mathbf{x},\tilde{B})^{\times}$. Note that if $P=x_i$ for some $1\leq i\leq n$, then $M(P)=x_i$, so that $\frac{\lambda P}{M(P)}$ is invertible.

The above argument proves that every irreducible cluster algebra element is indeed contained in the denoted set. On the other hand, we have to prove that every element in the set is indeed irreducible. The initial cluster variables are irreducible by Theorem \ref{GLS}. Suppose that $P\in S(\mathbf{a})$ is an irreducible polynomial which is not divisible by any $x_i$. Moreover, assume that $P\notin S(\mathbf{a'})$ for all sequences $\mathbf{a}'\geq\mathbf{a}$. Suppose that we have a factorization 
\begin{align*}
\frac{\lambda P}{\mathbf{x}^{\mathbf{a}}}=\frac{\mu F}{\mathbf{x}^{\mathbf{b}}} \cdot \frac{\nu G}{\mathbf{x}^{\mathbf{c}}}
\end{align*}
with $F\in S(\mathbf{b})$ and $Q\in S(\mathbf{c})$ for some sequences $\mathbf{b},\mathbf{c}\in \mathbf{N}^n$ and $\mu,\nu\in \mathcal{A}(\mathbf{x},\tilde{B})^{\times}$. First of all, note that neither $F$ nor $G$ is divisible by some $x_i$ with $1\leq i\leq n$. Furthermore, without loss of generality we may also assume that they are not divisible by any $x_i$ with $n+1\leq i\leq m$. We see that $P=\alpha FG$ for some scalar $\alpha\in K^{\times}$. By the irreducibility of $P$ we know that either $F$ or $G$ -- say $F$ --  is constant. But none of the ideals $I_i$ with $1\leq i\leq n$ contains $1$, so it follows that $\mathbf{b}=0$, so that $\frac{\mu F}{\mathbf{x}^{\mathbf{B}}}\in \mathcal{A}(\mathbf{x},\tilde{B})^{\times}$.

For the unique factorization, assume that 
\begin{align*}
\mathbf{x}^{\mathbf{a}} \cdot \frac{P_1}{\mathbf{x}^{\mathbf{a}_1}} \cdot \frac{P_2}{\mathbf{x}^{\mathbf{a}_2}} \cdots \frac{P_r}{\mathbf{x}^{\mathbf{a}}_r} =  \mathbf{x}^{\mathbf{b}} \cdot \frac{Q_1}{\mathbf{x}^{\mathbf{b}_1}} \cdot \frac{Q_2}{\mathbf{x}^{\mathbf{b}_2}} \cdots \frac{Q_s}{\mathbf{x}^{\mathbf{b}_s}} 
\end{align*} 
for two integers $r,s\geq 0$, some irreducible polynomials $P_i\in S(\mathbf{a}_i)$ and $Q_j\in S(\mathbf{b}_j)$ for $1\leq i\leq r$ and $1\leq j\leq s$, respectively, such that no $P_i$ or $Q_j$ is divisible by some $x_k$ with $1\leq k\leq m$ and such that $M(P_i)=\mathbf{x}^{\mathbf{a}_i}$ and $M(Q_j)=\mathbf{x}^{\mathbf{b}_j}$ for all $i,j$. 

Then $r=s$ and the $P_i$ are (up to multiplication with scalars) a permutation of the $Q_j$. Without loss of generality we have $P_i=Q_i$ for $1\leq i\leq r$. Then also $\mathbf{a}_i=\mathbf{b}_i$ for all $i$. It follows that $\mathbf{a}=\mathbf{b}$. Hence, every element in $\mathcal{A}(\mathbf{x},\tilde{B})$ admits a unique factorization in irreducible elements.
\end{proof}

\begin{cor} Thanks to Gei\ss-Leclerc-Schr\"oer's theorem \ref{GLS} cluster variables are irreducible. Thus, if we write a cluster variable as a Laurent polynomial $P(\mathbf{x})/M(\mathbf{x})$ in the cluster variables for a seed for which Conjecture \ref{IdealProperties} holds, then $P\in R$ is an irreducible polynomial.
\end{cor}

\section{Applications of the criterion and the Dynkin cases}
\label{Applications}

\subsection{Coefficient-free cluster algebras of rank 2} 
\label{RankTwo}

In this section we study classes of cluster algebras where Conjecture \ref{IdealProperties} is true. We only treat coefficient-free cluster algebras. We are already well-prepared to treat cluster algebras of rank 2. Recall that the coprimality of two ideals $I,J\subseteq R$ implies $I\cap J=IJ$.

Let $\mathcal{A}$ be a cluster algebra of rank 2 without frozen variables. In this case,
\begin{align*}
\tilde{B}=B=\begin{pmatrix}0&b\\-c&0\end{pmatrix}
\end{align*}
for some positive integers $b,c$. We also write $\mathcal{A}(b,c)$ instead of $\mathcal{A}(\mathbf{x},\tilde{B})$. Every such cluster algebra is acyclic. Moreover, we have $f_1=1+x_2^c$ and $f_2=1+x_1^b$ which are different as $(b,c)\neq(0,0)$.

\begin{prop} 
\label{Rank2}
If $f_1$ and $f_2$ are both irreducible polynomials in $K[x_1,x_2]$, then $\mathcal{A}(b,c)$ is a unique factorization domain. 
\end{prop}

\begin{proof}
Suppose that $f_1$ and $f_2$ are both irreducible in $K[x_1,x_2]$. By Theorem \ref{Criterion} it is enough to show that $I_1^{a_1}I_2^{a_2}=I_1^{a_1}\cap I_2^{a_2}$ for all $a_1,a_2\in\mathbb{N}$. The identity is exactly the coprimality of $I_1$ and $I_2$ which holds true by Proposition \ref{coprime}. 
\end{proof}

\subsection{Two lemmata}
\label{Lemmata}

In this subsection we state two lemmata which are crucial for various inductive proofs in the next subsection. The first lemma concerns sources and sinks. 
 
\begin{lemma}
\label{SinksNSources}
Let $\mathbf{a}\in\mathbb{N}^n$. Suppose that $I^{\mathbf{b}}=I_1^{b_1}\cap I_2^{b_2}\cap\ldots\cap I_n^{b_n}$ holds for all $\mathbf{b}\in\mathbb{N}^n$ such that $\sum_{i=1}^nb_i<\sum_{i=1}^na_i$. Suppose that the index $i\in\{1,2,\dots,n\}$ is either a source or a sink such that $a_i\neq 0$. Furthermore, assume that there exists an index $1\leq j\leq n$ with $a_j\neq 0$ such that $i$ and $j$ are adjacent. Then we also have $I^{\mathbf{a}}=I_1^{a_1}\cap I_2^{a_2}\cap\ldots\cap I_n^{a_n}$.
\end{lemma}

\begin{proof} Let $P$ be a polynomial in the intersection $I_1^{a_1}\cap I_2^{a_2}\cap\ldots I_n^{a_n}$. We have to show that $P\in I^{\mathbf{a}}$. By Proposition \ref{coprime} the ideals $I_i$ and $I_j$ are coprime and hence there exist polynomials $F\in I_i^{a_i}$ and $I_j^{a_j}$ such that $F+G=1$. By assumption we know that
\begin{align*}
P\in I_1^{a_1}\cap\ldots \cap \widehat{I_i^{a_i}}\cap\ldots \cap I_n^{a_n}=I_1^{a_1}\cdots \widehat{I_i^{a_i}}\cdots I_n^{a_n},\\
P\in I_1^{a_1}\cap\ldots \cap \widehat{I_j^{a_j}}\cap\ldots \cap I_n^{a_n}=I_1^{a_1}\cdots \widehat{I_j^{a_j}}\cdots I_n^{a_n}.
\end{align*}
We conclude that $PF,PG\in I^{\mathbf{a}}$ and hence $P=P(F+G)\in I^{\mathbf{a}}$.
\end{proof}

For the second lemma we introduce the notation $N(i)=\{j\colon 1\leq j\leq n, b_{ij}\neq 0\}$ for the set of \textit{mutable neigbors} of a mutable or frozen index $i\in\{1,2,\ldots,m\}$.   

\begin{lemma} 
\label{Free}
Let $\mathbf{a}\in\mathbb{N}^n$. Suppose that $i$ is a mutable index with $a_i\neq 0$ and that $I^{\mathbf{b}}=I_1^{b_1}\cap I_2^{b_2}\cap\ldots \cap I_n^{b_n}$ holds for all $\mathbf{b}\in\mathbb{N}^n$ with $\sum_{j=1}^nb_j<\sum_{j=1}^na_j$. Assume that one of the following conditions holds.
\begin{itemize}
\item[(a)] For all indices $j\in N(i)$ we have $a_j=0$.
\item[(b)] The initial exchange polynomial $f_i$ has the form $f_i=x_k+M_i$ for some (mutable or frozen) index $k$ and some monomial $M_i\in R$. Furthermore, for all neighbors $j\in N(k)\backslash \{i\}$ we have $a_j=0$.
\end{itemize}
Then we also have $I^{\mathbf{a}}=I_1^{a_1}\cap I_2^{a_2}\cap \ldots \cap I_n^{a_n}$.
\end{lemma}

\begin{proof} Let $P$ be a polynomial in the intersection $I_1^{a_1}\cap I_2^{a_2}\cap \ldots \cap I_n^{a_n}$. We have to show that $P\in I^{\mathbf{a}}$. We consider the sequence $\mathbf{a}'\in\mathbb{N}^n$ be the element with $a'_i=0$ and $a'_j=a_j$ for $j\neq i$. The assumption $a_i\neq 0$ implies that  
\begin{align*}
P\in I_1^{a_1}\cap\ldots \cap \widehat{I_i^{a_i}}\cap\ldots \cap I_n^{a_n}=I^{\mathbf{a}'}.
\end{align*}

Assume that $(a)$ holds. We can write 
\begin{align*}
P=\sum_{0\leq\mathbf{b}\leq\mathbf{a}'}P_{\mathbf{b}}x^{\mathbf{b}}f^{\mathbf{a}'-\mathbf{b}}
\end{align*} 
for some polynomials $P_{\mathbf{b}}\in R$. The assumption $(a)$ implies that the polynomials $x^{\mathbf{b}}f^{\mathbf{a}'-\mathbf{b}}$ do not depend on $x_i$. Let us order the remaining terms by powers of $x_i$, i.e., let us write $P_{\mathbf{b}}=\sum_{r=0}^{\infty} P_{\mathbf{b},r}x_i^r$ for some (uniquely determined) polynomials $P_{\mathbf{b},r}\in K[x_1,\ldots,\widehat{x_i},\ldots,x_m]$. It follows from Proposition \ref{OneIdeal} that $f_i^{a_i-r}\vert \sum_{0\leq\mathbf{b}\leq\mathbf{a}'} P_{\mathbf{b},r}x^{\mathbf{b}}f^{\mathbf{a}'-\mathbf{b}}$ for all $0\leq r\leq a_i$. Let $Q_r$ be the quotient of those two polynomials. Then $f_i^{a_i-r}Q_r\in I_j^{a_j}$ for all mutable indices $j\neq i$. But Proposition \ref{InitialPolyNotInOtherIdeal} asserts that $f_i\notin I_j$ for all $i\neq j$, which together with Lemma~\ref{Cancelling} yields $Q_r\in I_j^{a_j}$ for all $j\neq i$. We conclude that $P=\sum_{r=0}^{\infty}x_i^rf_i^{a_i-r}Q_r\in I^{\mathbf{a}}$.

Assume that $(b)$ holds. The argument is basically the same. This time we write 
\begin{align*}
P=\sum_{0\leq\mathbf{b}\leq\mathbf{a}'\atop r\geq 0}P_{\mathbf{b},r}f_i^rx^{\mathbf{b}}f^{\mathbf{a}'-\mathbf{b}}
\end{align*} 
with $P_{\mathbf{b},r}\in K[x_1,\ldots,\widehat{x_k},\ldots,x_m]$. The assumption $(b)$ implies that the polynomials $x^{\mathbf{b}}f^{\mathbf{a}'-\mathbf{b}}$ in the sum do not depend on $x_k$. Moreover, Remark \ref{ViceVersa} implies $x_i^{a_i-r}\vert \sum_{0\leq\mathbf{b}\leq\mathbf{a}'}P_{\mathbf{b},r}x^{\mathbf{b}}f^{\mathbf{a}'-\mathbf{b}}$ for $0\leq r\leq a_i$. A similar argument as above together with the observation $x_i\notin I_j$ for $i\neq j$ finishes the proof.\end{proof}

\subsection{The Dynkin cases}
\label{DynkinUFD}
As promised, let us now consider coefficient-free cluster algebras of finite type. As in Section \ref{BN4} let $\Delta$ be a Dynkin diagram of type $A, D$ or $E$. Example \ref{Dynkin} shows that $\mathcal{A}(\mathbf{x},\Delta)$ is not a unique factorization domain for $\Delta$ of type $A_3$ or $D_n(n\geq 4)$. 

It turns out that in all other cases, namely $A_n (n\neq 1,3)$ and $E_n(n=6,7,8)$, there exists an orientation $Q$ of $\Delta$ such that the initial exchange polynomials with respect to the initial seed $(\mathbf{x},B(Q))$ are irreducible and pairwise different and such that Conjecture \ref{IdealProperties} holds. Thus, the cluster algebras $\mathcal{A}(\mathbf{x},\Delta)$ are unique factorization domains in these cases. The results are quite surprising, because other ring theoretic properties behave differently, see for example Muller's results on regularity \cite[Section 7]{M} which in type $A_n$ depend on whether $n\equiv 3$ $(\mathrm{mod} \ 4)$.

\begin{theorem} Let $n\neq 1,3$. Moreover, let $Q$ be the linear orientation of the Dynkin diagram $A_n$, i.e., $Q_0=\{1,2,\ldots,n\}$ and $Q_1=\{i\rightarrow i+1 \colon 1\leq i\leq n-1\}$. Then $I^{\mathbf{a}}=I_1^{a_1}\cap I_2^{a_2}\cap\ldots\cap I_n^{a_n}$ holds for all sequences $\mathbf{a}\in\mathbb{N}^n$.
\end{theorem}

\begin{proof} The case $n=2$ is treated in Proposition \ref{Rank2}, so let us assume that $n\geq 4$. Let $P$ be a polynomial in the intersection $I_1^{a_1}\cap I_2^{a_2}\cap\ldots I_n^{a_n}$. We have to show that $P\in I^{\mathbf{a}}$ which we prove by mathematical induction on $\sum_{i=1}^na_i$. The base case is trivial. Assume that the statements holds for all sequences with a smaller sum. If $a_1$ and $a_2$ are both greater than $0$, then the claim follows from Lemma \ref{SinksNSources}, because $i=1$ is a source which is adjacent to $j=2$. If $a_1$ is greater than $0$, but $a_2=0$, then the claim follows from Lemma \ref{Free} (a), because $N(1)=\{2\}$. If $a_1=0$, then we apply Lemma \ref{Free} (b) with $i$ the smallest integer such that $a_i\neq 0$ (which is necessarily greater than $1$) and $k=i-1$. It is easy to see that assumption (b) is satisfied because either $i=2$ or $a_{i-2}=0$.     
\end{proof}

\begin{theorem} Let $n\in\{6,7,8\}$. Moreover, let $Q$ be the orientation of $E_n$ with arrows $1\rightarrow 2\rightarrow 3$, $n\rightarrow \ldots\rightarrow 6\rightarrow 5\rightarrow 3$ and $3\rightarrow 4$. Then $I^{\mathbf{a}}=I_1^{a_1}\cap I_2^{a_2}\cap\ldots\cap I_n^{a_n}$ holds for all sequences $\mathbf{a}\in\mathbb{N}^n$.
\end{theorem}

\begin{figure}
\begin{center}
\begin{tikzpicture}
\node[rectangle,rounded corners,draw] (1) at (2,0) {$1$};
\node[rectangle,rounded corners,draw] (2) at (0,0) {$2$};
\node[rectangle,rounded corners,draw] (3) at (0,-2) {$3$};
\node[rectangle,rounded corners,draw] (4) at (2,-2) {$4$};
\node[rectangle,rounded corners,draw] (5) at (0,-4) {$5$};
\node[rectangle,rounded corners,draw] (6) at (2,-4) {$6$};
\draw[->, >=latex', shorten >=2pt, shorten <=2pt,thick] (1) to node {} (2);
\draw[->, >=latex', shorten >=2pt, shorten <=2pt,thick] (2) to node {} (3);
\draw[->, >=latex', shorten >=2pt, shorten <=2pt,thick] (3) to node {} (4);
\draw[->, >=latex', shorten >=2pt, shorten <=2pt,thick] (5) to node {} (3);
\draw[->, >=latex', shorten >=2pt, shorten <=2pt,thick] (6) to node {} (5);
\end{tikzpicture}
\end{center}
\caption{The orientation of the Dynkin diagram $E_6$}
\label{ESix}
\end{figure}
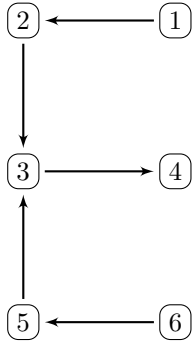

\begin{proof} As in the proof of the previous theorem we proceed by induction on $\sum_{i=1}^na_i$. The base case is trivial. Assume that the statements holds for all sequences with a smaller sum. If $a_3$ and $a_4$ are both greater than $0$, then the claim follows from Lemma \ref{SinksNSources}, because $i=4$ is a sink which is adjacent to $j=3$. If $a_4$ is greater than $0$, but $a_3=0$, then the claim follows from Lemma \ref{Free} (a) with $i=4$, because $N(4)=\{3\}$. If $a_3$ is greater than $0$, but $a_4=0$, then the claim follows from Lemma \ref{Free} (b) with $i=3$ and $k=4$, because $N(4)=\{3\}$. So we may assume $a_3=a_4=0$. In this case, we have $I_1=(x_1,1+x_2)$, $I_2=(x_2,x_1+x_3)$, $I_i=(x_i,x_{i-1}+x_{i+1})$ for $5\leq i \leq n-1$, and $I_n=(x_n,1+x_{n-1})$, which are of the same form as in type $A_{n-1}$. Essentially the same argument yields $I^{\mathbf{a}}=I_1^{a_1}\cap I_2^{a_2}\cap\ldots\cap I_n^{a_n}$. 
\end{proof}

\subsection*{Acknowledgments.} The author would like to thank Christof Gei\ss, Anurag Singh and Jan Schr\"oer for helpful discussions and Matthias Warkentin for comments on an earlier version of the article.

\end{document}